\documentclass[11pt,reqno]{amsart}
\usepackage[T1]{fontenc}
\usepackage[utf8]{inputenc}
\usepackage{amsmath,amssymb,amsthm,mathtools,microtype}
\usepackage[a4paper,margin=30mm]{geometry}
\usepackage[hidelinks]{hyperref}
\allowdisplaybreaks[2]
\theoremstyle{plain}
\newtheorem{theorem}{Theorem}[section]
\newcommand{\Dq}{D_q}
\newcommand{\Rq}{\mathbb R_{q,+}}
\newcommand{\Sq}{S_q}
\newcommand{\Eqsc}{\mathcal E_q}
\newcommand{\LCM}{q\text{-LCM}}
\newcommand{\CM}{q\text{-CM}}
\newcommand{\BF}{q\text{-BF}}
\newcommand{\Lq}{\mathcal L_q}
\newcommand{\dd}{\,\mathrm d}
\DeclareMathOperator{\Li}{Li}
\title[$q$-Bernstein Functions and Applications]{$q$-Bernstein Functions and Applications}
\author{Valmir Krasniqi}
\address{Institute of Science, Technology, Engineering and Mathematics (STEM), Rr. Nekibe Kelmendi Nr. 26, 10000 Prishtin\"e, Republic of Kosovo}
\email{valmir@stem-ks.org}
\author{Toufik Mansour}
\address{Department of Mathematics, University of Haifa, 3498838 Haifa, Israel}
\email{tmansour@univ.haifa.ac.il}
\begin{document}
\begin{abstract}
We characterize $q$-Bernstein functions by exponential tests and, on a geometric
lattice, by a well-defined spectral $q$-Laplace convolution semigroup. We also
establish results on $q$-completely monotonic, $q$-log-completely monotonic,
and $q$-Bernstein functions and give applications to $q$-gamma functions.
\end{abstract}
\maketitle

\section{Introduction}
Throughout the paper $0<q<1$, except where the $q$-gamma function is
explicitly considered for $q>1$. Write
\[
 [x]=[x]_q=\frac{1-q^x}{1-q},\qquad
 [n]!=\prod_{j=1}^n[j],\qquad
 \genfrac{[}{]}{0pt}{}{n}{k}_q=\frac{[n]!}{[k]![n-k]!}.
\]
The $q$-Pochhammer symbol and the two $q$-exponentials are
\[
 (a;q)_\infty=\prod_{j=0}^\infty(1-aq^j),\quad
 e_q(z)=\sum_{n=0}^\infty\frac{z^n}{[n]!},\quad
 E_q(z)=\sum_{n=0}^\infty\frac{q^{n(n-1)/2}z^n}{[n]!}
      =(-(1-q)z;q)_\infty.
\]
The first series converges for $|z|<(1-q)^{-1}$, while the second is entire.
As in the original convention, put
\[
 \Eqsc^u=(E_q(1))^u,\quad
 \operatorname{Log}_q v=\frac{\log v}{\log E_q(1)},\quad
 c_q=\log E_q(1)>0.
\]
In particular, the scalar exponential $\Eqsc^u$ is not the entire
$q$-exponential $E_q(u)$. For a function $u$ on a $q$-stable set
$I\subseteq(0,\infty)$, the Jackson derivative is
\[
 \Dq u(x)=\frac{u(x)-u(qx)}{(1-q)x}.
\]
Its iterates are finite linear combinations of values on the orbit
$\{x,qx,q^2x,\ldots\}$. The identities
\[
 \Dq^n e_q(ax)=a^ne_q(ax),\qquad
 \Dq^n E_q(ax)=a^nq^{n(n-1)/2}E_q(aq^nx)
\]
hold wherever the terms on the left are defined (the first identity
in the indicated convergence disk or its valid continuation).

A nonnegative function $u$ is called $q$-completely monotonic ($\CM$) if
\[
 (-1)^n\Dq^n u(x)\ge0\qquad(n\ge0,x\in I).
\]
A strictly positive function $u$ is $q$-log-completely monotonic
($\LCM$) if
\[
 (-1)^n\Dq^n\operatorname{Log}_q u(x)\ge0\qquad(n\ge1,x\in I).
\]
A nonnegative function $f$ is a $q$-Bernstein function ($\BF$) if
\[
 (-1)^{n-1}\Dq^n f(x)\ge0\qquad(n\ge1,x\in I),
\]
equivalently $\Dq f$ is $\CM$. These are Jackson-sign conditions; they
do not by themselves imply the corresponding signs of ordinary derivatives.
The classical notions of complete monotonicity and of Bernstein functions
use ordinary derivatives; classical Bernstein functions admit the usual
L\'evy--Khintchine representation \cite{Schilling}.

\section{Main results}
We first recall the elementary observation that if $f$ is nondecreasing
(respectively nonincreasing) on $(0,\infty)$, then $\Dq f\ge0$
(respectively $\Dq f\le0$). For strict monotonicity the respective
inequality is strict. This follows immediately from $qx<x$.

\begin{theorem}\label{thm21}
If $f$ is a classical log-completely monotonic function, then $f$ is
$q$-log-completely monotonic.
\end{theorem}
\begin{proof}
For $n\ge1$, let $x_j=q^jx$ ($0\le j\le n$). For divided
differences at these nodes the identity
\[
 \Dq^n v(x)=[n]!\,v[x_0,x_1,\ldots,x_n]
\]
follows from the divided-difference recursion, or from the finite
$q$-binomial formula. The Hermite--Genocchi formula gives
\[
 v[x_0,\ldots,x_n]
 =\int_{\{t_i\ge0,\,\sum_{i=0}^nt_i=1\}}
 v^{(n)}\!\left(\sum_{i=0}^nt_ix_i\right)\dd t_1\cdots\dd t_n.
\]
Thus the sign of $v^{(n)}$ is preserved by $\Dq^n$.
Apply this to $v=\operatorname{Log}_q f=(\log f)/c_q$.
Since $c_q>0$, the classical logarithmic complete monotonicity
inequalities give the desired Jackson inequalities.
\end{proof}

For ordinary composition, the correct iterated Jackson chain rule uses
\emph{ordinary divided differences}, not the symbolic $q$-composition
appearing in the $q$-Fa\`a di Bruno formula of \cite{Johnson}.
Let $h[x_i,\ldots,x_j]$ denote an ordinary divided difference, put
$x_j=q^jx$, and interpret divided differences of $g$ at coincident
arguments in the Hermite sense. The correct identity is
\begin{equation}\tag{1}\label{eq:chain}
 \Dq^n(g\circ h)(x)
 =[n]!\sum_{k=1}^n\ \sum_{0=i_0<i_1<\cdots<i_k=n}
 g[h(x_{i_0}),\ldots,h(x_{i_k})]
 \prod_{j=1}^k h[x_{i_{j-1}},\ldots,x_{i_j}].
\end{equation}
It follows from the divided-difference recursion (or from the chain
rule for divided differences). In particular, for $g(t)=t^2$ and
$h(x)=x^2$ it gives $\Dq(g\circ h)(x)=[4]x^3$;
the ordinary-composition interpretation of the symbolic formula would
incorrectly give $[2]^2x^3$.

\begin{theorem}\label{thm22}
Let $f$ be any $q$-log-completely monotonic function. Then $f$ is a
$q$-completely monotonic function.
\end{theorem}
\begin{proof}
Put $h(x)=\log f(x)=c_q\operatorname{Log}_qf(x)$ and
$g(t)=e^t$. For each segment $x_i,x_{i+1},\ldots,x_j$,
\[
 (-1)^{j-i}h[x_i,\ldots,x_j]
 =\frac{c_q}{[j-i]!}(-1)^{j-i}
      \Dq^{j-i}\operatorname{Log}_q f(x_i)\ge0.
\]
Every divided difference of $g=e^t$, including repeated-node cases,
is strictly positive, by the Hermite--Genocchi formula. In each term
of (1) the inner factors together have sign
$(-1)^{\sum_j(i_j-i_{j-1})}=(-1)^n$.
Consequently $(-1)^n\Dq^nf(x)\ge0$ for every $n\ge1$.
The case $n=0$ follows from $f>0$.
\end{proof}

\begin{theorem}\label{thm23}
Let $f:I\to[0,\infty)$ on a $q$-stable set $I$. Then $f$ is a
$q$-Bernstein function if and only if $\Eqsc^{-tf}$ is
$q$-completely monotonic for every $t>0$.
\end{theorem}
\begin{proof}
If $f$ is $q$-Bernstein, then
\[
 \operatorname{Log}_q(\Eqsc^{-tf})=-tf
\]
has alternating Jackson derivatives of the signs required for
$q$-log-complete monotonicity. Theorem~\ref{thm22} gives the forward
implication. Conversely fix $x\in I$ and $n\ge1$.
Since $\Dq^n$ is a finite linear combination of evaluations,
\[
 0\le(-1)^n\Dq^n(\Eqsc^{-tf})(x)
 =t c_q(-1)^{n-1}\Dq^nf(x)+O_{x,n}(t^2)
 \qquad(t\downarrow0).
\]
Divide by $t c_q>0$ and let $t\downarrow0$. The resulting
inequalities prove that $f$ is $q$-Bernstein.
\end{proof}

\begin{theorem}\label{thm24}
Let $f:I\to[0,\infty)$ on a $q$-stable set $I$. Then
$g\circ f$ is $q$-completely monotonic for \emph{every classical}
completely monotonic function $g$ on $[0,\infty)$ if and only if
$\Eqsc^{-f}$ is $q$-log-completely monotonic.
\end{theorem}
\begin{proof}
Suppose first that $g\circ f$ is $q$-completely monotonic for
every such $g$. Apply this to the classical completely monotonic
functions $g(u)=e^{-tu}$, for all $t>0$. By
Theorem~\ref{thm23}, $f$ is $q$-Bernstein, so
$\operatorname{Log}_q(\Eqsc^{-f})=-f$ has the required signs.

Conversely, $q$-log-complete monotonicity of $\Eqsc^{-f}$ means
precisely that $f$ is $q$-Bernstein. By the classical Bernstein
representation of a completely monotonic function, there is a
positive measure $\rho$ such that
\[
 g(u)=\int_{[0,\infty)}e^{-su}\dd\rho(s)
 \qquad(u\ge0),
\]
where the integral is finite at the values under consideration.
Theorem~\ref{thm23} gives that $e^{-sf}$ is $q$-completely
monotonic for every $s\ge0$. For fixed $x,n$, the finite
combination defining $\Dq^n$ may be interchanged with the
pointwise finite integrals, yielding
\[
 (-1)^n\Dq^n(g\circ f)(x)
 =\int_{[0,\infty)}(-1)^n\Dq^n(e^{-sf})(x)\dd\rho(s)\ge0.
\]
This proves the assertion.
\end{proof}

\begin{theorem}\label{thm25}
Let $f$ be a $q$-Bernstein function and let $g$ be a
\emph{classical} Bernstein function, with $g$ defined on the range
of $f$. Then $g\circ f$ is a $q$-Bernstein function.
\end{theorem}
\begin{proof}
The classical L\'evy--Khintchine formula gives
\[
 g(u)=a+bu+\int_{(0,\infty)}(1-e^{-su})\dd\nu(s),
 \qquad a,b\ge0,\quad
 \int_{(0,\infty)}\min(1,s)\dd\nu(s)<\infty.
\]
For every $s>0$, Theorem~\ref{thm23} shows that $e^{-sf}$ is
$q$-completely monotonic. Hence $1-e^{-sf}$ is nonnegative and
$q$-Bernstein. The same is true of $a+bf$.
The integral for $g(f(x))$ is finite at every $x$; the same holds
at the finitely many points $q^jx$ used in any Jackson iterate.
Interchanging this finite linear combination with the integrals
and using the sign inequalities shows
\[
 (-1)^{n-1}\Dq^n(g\circ f)(x)\ge0,\qquad n\ge1.
\]
Also $g\circ f\ge0$.
\end{proof}

\begin{theorem}\label{thm26}
\begin{enumerate}
\item[\textup{(i)}] Let $f\ge0$ be a \emph{classical} completely monotonic
function on $(0,\infty)$, and let $a>0$. Then $x\mapsto f(x)-f(x+a)$
is $q$-completely monotonic.
\item[\textup{(ii)}] Let $f\ge0$ be a $C^1$ function on $(0,\infty)$.
If $f(x)-f(x+a)$ is $q$-completely monotonic for each $a$ in some
right-hand neighborhood of $0$, then $f$ is $q$-completely monotonic.
\end{enumerate}
\end{theorem}
\begin{proof}
(i) A classical completely monotonic $f$ has a Laplace representation
$f(x)$ has the representation
\[
 f(x)=\int_{[0,\infty)}e^{-xt}\dd\rho(t).
\]
Thus
\[
 f(x)-f(x+a)=\int_{[0,\infty)}e^{-xt}(1-e^{-at})\dd\rho(t)
\]
is itself classically completely monotonic. It is therefore
$q$-completely monotonic by the divided-difference argument of
Theorem~\ref{thm21}.

(ii) Each positive multiple $\bigl(f(x)-f(x+a)\bigr)/a$ is
$q$-completely monotonic. Pointwise, as $a\downarrow0$, it
converges to $-f'(x)$. Since each $\Dq^n$ uses only finitely many
evaluations, the $q$-completely monotonic inequalities pass to
this limit: $-f'$ is $q$-completely monotonic. For $n\ge1$,
the fundamental theorem of calculus and the scaling identity
$\Dq(v(sx))=s(\Dq v)(sx)$ give
\[
 \Dq^n f(x)=\frac1{1-q}\int_q^1s^{n-1}
  (\Dq^{n-1}f')(sx)\dd s.
\]
Consequently $(-1)^n\Dq^n f(x)\ge0$ for all $n\ge1$.
The order-zero condition is the assumed $f\ge0$.
\end{proof}

For the next result define the geometric lattice and spectral set by
\[
 \Rq=\{q^m:m\in\mathbb Z\},\qquad
 \Sq=\left\{\frac{q^k}{1-q}:k\in\mathbb Z\right\},\qquad
 \Sq^0=\Sq\cup\{0\}.
\]
For $s\in\Sq^0$ and $x\in\Rq$, set
\[
 K_s(x)=E_q(-sx),\qquad K_0(x)=1,\qquad K_s(0)=1,
\]
and for a finite positive measure $\mu$ on $\Sq^0$, define its
$q$-Laplace transform
\[
 \Lq\mu(x)=\int_{\Sq^0}K_s(x)\dd\mu(s).
\]
The measure convolution $\star_q$ is defined by the transform identity
\[
 \Lq(\mu\star_q\nu)(x)=\Lq\mu(x)\,\Lq\nu(x).
\]
Existence, positivity, uniqueness and associativity of this operation
are established in the proof of the following theorem. In particular,
this is \emph{not} convolution under ordinary addition, since $\Rq$
is not closed under ordinary addition.

\begin{theorem}\label{thm27}
A function $f:\Rq\cup\{0\}\to[0,\infty)$ that is continuous at
$0$ along $\Rq$ and satisfies $f(0)=0$ is $q$-Bernstein on $\Rq$
if and only if there exists a unique family $(\mu_t)_{t\ge0}$ of
probability measures on $\Sq^0$, weakly continuous at $t=0$,
with
\begin{equation}\tag{2}\label{eq:semigroup}
 \mu_0=\delta_0,\qquad
 \mu_t\star_q\mu_s=\mu_{t+s},\qquad
 \Lq\mu_t(x)=\Eqsc^{-tf(x)}\quad(t,s\ge0,\ x\in\Rq).
\end{equation}
Conversely, every such weakly continuous $\star_q$-convolution
semigroup of probability measures has an exponent $f$ of this form.
\end{theorem}
\begin{proof}
We first justify the spectral transform and convolution used here.
Write $C_q=(q;q)_\infty>0$ and $s=q^k/(1-q)$, $x=q^m$. Then
\[
 K_s(q^m)=\begin{cases}
 \dfrac{C_q}{(q;q)_{k+m-1}},&k+m\ge1,\\[5pt]
 0,&k+m\le0.
 \end{cases}
\]
In particular $0\le K_s(x)\le1$, and
\[
 (-1)^n\Dq^nK_s(x)
   =s^nq^{n(n-1)/2}K_s(q^nx)\ge0.
\]
We show that a finite-valued function $u:\Rq\to[0,\infty)$ is
$q$-completely monotonic precisely when it has a \emph{unique}
representation $u(x)=a+\int_{\Sq}K_s(x)\dd\rho(s)$,
with $a\ge0$ and a positive measure $\rho$ (the integral finite for
each $x$). Indeed, if $u$ is $q$-completely monotonic, put
\[
 u_m=u(q^m),\qquad a=\lim_{m\to-\infty}u_m,\qquad
 h_m=u_m-a,\qquad (Th)_m=h_{m-1}.
\]
The sequence $u_m$ is increasing in $m$ and $0\le h_{m-j}\le h_m$.
The Jackson sign inequalities imply, for all $N\ge1$,
\[
 \left[\prod_{r=0}^{N-1}(I-q^rT)h\right]_m
 =(-1)^N(1-q)^Nq^{Nm-N(N+1)/2}\Dq^Nu(q^{m-N})\ge0.
\]
The finite $q$-binomial formula and dominated convergence give
\[
 v_m:=\sum_{j=0}^\infty
  \frac{(-1)^jq^{j(j-1)/2}}{(q;q)_j}h_{m-j}\ge0.
\]
Setting $p_m=h_m-h_{m-1}$, the Euler product and its reciprocal,
whose coefficients are summable in the relevant backward shifts, give
\[
 v=(qT;q)_\infty p,\qquad
 p_m=\sum_{j=0}^\infty\frac{q^j}{(q;q)_j}v_{m-j}.
\]
As $h_m=\sum_{r\ge0}p_{m-r}$, Tonelli's theorem and
$\sum_{j=0}^{\ell}q^j/(q;q)_j=1/(q;q)_\ell$ yield
\[
 h_m=\sum_{\ell=0}^\infty\frac{v_{m-\ell}}{(q;q)_\ell}.
\]
Thus the unique masses are
\[
 \rho\!\left(\left\{\frac{q^k}{1-q}\right\}\right)
 =\frac{v_{1-k}}{C_q},\qquad k\in\mathbb Z,
\]
and $u(x)=a+\int K_s(x)\dd\rho(s)$. Conversely, these kernels
are $q$-completely monotonic, and the finite-difference operator
passes through every pointwise finite integral. Finally,
$K_s(x)\uparrow1$ as $x\to0$ along the lattice, so a finite lattice
limit of $u$ equals $a+\rho(\Sq)$, by monotone convergence.
This proves the stated existence, uniqueness and normalization.

The product of two $q$-completely monotonic functions is
$q$-completely monotonic: the $q$-Leibniz formula is
\[
 \Dq^n(uv)(x)=\sum_{j=0}^n\genfrac{[}{]}{0pt}{}{n}{j}_q
  (\Dq^{n-j}u)(q^jx)(\Dq^jv)(x),
\]
and every summand has the required sign.
Therefore the product of the transforms of two finite positive measures
on $\Sq^0$ has a unique finite positive representing measure
(because its limit at zero is the product of the two total masses).
This defines $\star_q$. Uniqueness of representation proves its
commutativity, associativity, bilinearity and identity $\delta_0$;
its total mass is multiplicative. Thus probabilities are closed
under $\star_q$.

Now let $f$ be $q$-Bernstein and continuous at zero with $f(0)=0$.
By Theorem~\ref{thm23}, $\Eqsc^{-tf}$ is $q$-completely monotonic
for each $t\ge0$, and its lattice limit at zero is $1$.
It consequently has a unique representing probability measure
$\mu_t$ on $\Sq^0$. Multiplicativity of the scalar exponential and
uniqueness imply the semigroup law in (2).
For $s\in\Sq$ and $x\in\Rq$ we also have the elementary bound
\[
 1-K_s(x)\ge\min\{1,(1-q)sx\}.
\]
For any $\varepsilon>0$ and fixed $x\in\Rq$ this gives
\[
 \mu_t([\varepsilon,\infty))
 \le\frac{1-\Eqsc^{-tf(x)}}{\min\{1,(1-q)\varepsilon x\}}
 \longrightarrow0\qquad(t\downarrow0),
\]
which proves $\mu_t\Rightarrow\delta_0$.

Conversely, given a weakly continuous probability semigroup,
$h_x(t)=\Lq\mu_t(x)$ is continuous at zero, takes values in $[0,1]$,
satisfies $h_x(t+s)=h_x(t)h_x(s)$ and has $h_x(0)=1$.
It is strictly positive, since $h_x(t)=h_x(t/n)^n$ and
$h_x(t/n)\to1$. The continuous multiplicative Cauchy equation
yields $h_x(t)=e^{-t c_q f(x)}$ for a unique $f(x)\ge0$.
Since $K_s(x)\to1$ as $x\to0$ for every $s\in\Sq^0$,
dominated convergence gives $f(0)=0$ and lattice continuity.
Each $\Eqsc^{-tf}=\Lq\mu_t$ is $q$-completely monotonic;
Theorem~\ref{thm23} implies that $f$ is $q$-Bernstein.
\end{proof}

A strictly positive $q$-completely monotonic function $g$ is
called $q$-infinitely divisible if $g^t$ is $q$-completely monotonic
for every $t>0$.

\begin{theorem}\label{thm28}
Let $g:(0,\infty)\to(0,1]$. Then $g$ is $q$-infinitely divisible
if and only if $g=\Eqsc^{-f}$ for some $q$-Bernstein function $f$.
\end{theorem}
\begin{proof}
Suppose $g$ is $q$-infinitely divisible, and define
$f=-\operatorname{Log}_q g\ge0$. Fix $x>0$ and $n\ge1$.
Expanding at the finitely many orbit points gives
\[
 0\le(-1)^n\Dq^n(g^t)(x)
 =t(-1)^n\Dq^n\log g(x)+O_{x,n}(t^2).
\]
Division by $t>0$ and passage to zero yields
$(-1)^{n-1}\Dq^nf(x)\ge0$, so $f$ is $q$-Bernstein
and $g=\Eqsc^{-f}$. Conversely, if $g=\Eqsc^{-f}$ with $f$
$q$-Bernstein, then $g^t=\Eqsc^{-tf}$ is $q$-completely
monotonic for every $t>0$ by Theorem~\ref{thm23}.
\end{proof}

\begin{theorem}\label{thm29}
Let $g:(0,\infty)\to(0,\infty)$. Then $g$ is
$q$-infinitely divisible if and only if $g$ is
$q$-log-completely monotonic.
\end{theorem}
\begin{proof}
If $g^t$ is $q$-completely monotonic for all $t>0$,
then for every $n\ge1$ the finite-evaluation expansion
\[
 0\le(-1)^n\Dq^n(g^t)(x)
 =t(-1)^n\Dq^n\log g(x)+O_{x,n}(t^2)
\]
implies $(-1)^n\Dq^n\operatorname{Log}_qg\ge0$.
Conversely, $g$ $q$-log-completely monotonic implies that each
$g^t$ is $q$-log-completely monotonic, since
$\operatorname{Log}_q(g^t)=t\operatorname{Log}_qg$.
Theorem~\ref{thm22} makes each $g^t$ $q$-completely monotonic.
\end{proof}

\section{Applications}
For $0<q<1$ the $q$-gamma function is
\begin{equation}\tag{3}\label{eq:qgamma-low}
 \Gamma_q(x)=\frac{(q;q)_\infty}{(q^x;q)_\infty}
 (1-q)^{1-x},\qquad x>0,
\end{equation}
whereas for $q>1$ it is
\begin{equation}\tag{4}\label{eq:qgamma-high}
 \Gamma_q(x)=\frac{(q^{-1};q^{-1})_\infty}
 {(q^{-x};q^{-1})_\infty}(q-1)^{1-x}q^{x(x-1)/2}.
\end{equation}
As usual $\psi_q(x)=\dfrac{\mathrm d}{\mathrm dx}\log\Gamma_q(x)$.
The product (3) gives, for $0<q<1$,
\begin{equation}\tag{5}\label{eq:psi-low}
 \psi_q(x)=-\log(1-q)+\log q
  \sum_{n=1}^\infty\frac{q^{nx}}{1-q^n}.
\end{equation}
For $q>1$, formula (4) gives
\begin{equation}\tag{6}\label{eq:psi-high}
 \psi_q(x)=-\log(q-1)+\log q
 \left(x-\frac12-\sum_{n=1}^\infty
 \frac{q^{-nx}}{1-q^{-n}}\right).
\end{equation}
In either regime, $\psi_q'$ is strictly classically completely monotonic.
For $0<q<1$ and $k\ge1$,
\begin{equation}\tag{7}\label{eq:psi-deriv-low}
 \psi_q^{(k)}(x)=(\log q)^{k+1}
 \sum_{n=1}^\infty\frac{n^kq^{nx}}{1-q^n}.
\end{equation}
For $q>1$, the correct derivatives are
\begin{equation}\tag{8}\label{eq:psi-deriv-high1}
 \psi_q'(x)=\log q+(\log q)^2
 \sum_{n=1}^\infty\frac{nq^{-nx}}{1-q^{-n}},
\end{equation}
and for $k\ge2$,
\begin{equation}\tag{9}\label{eq:psi-deriv-highk}
 \psi_q^{(k)}(x)=(-1)^{k-1}(\log q)^{k+1}
 \sum_{n=1}^\infty\frac{n^kq^{-nx}}{1-q^{-n}}.
\end{equation}
All differentiations of these absolutely convergent series are
justified locally uniformly on $(0,\infty)$.

The polylogarithm is $\Li_s(z)=\sum_{k=1}^\infty z^k/k^s$ for
$|z|<1$. For real $\alpha$ and $\beta\ge0$ we consider the
function in the original application, with its exponent sign corrected:
\begin{equation}\tag{10}\label{eq:application}
 f_{\alpha,\beta,q}(x)
 =(1-q)^x\frac{e^{-h(x)}\Gamma_q(x+\beta)}
 {[x]^{x+\beta-\alpha}},\qquad
 h(x)=-\frac{\Li_2(q^x)+x\log q\log(1-q^x)}{\log q}.
\end{equation}

\begin{theorem}\label{thm31}
Let $2\alpha\le1\le\beta$ and $0<q<1$. Then the function
$f_{\alpha,\beta,q}$ in (10) is $q$-log-completely monotonic
on $(0,\infty)$.
\end{theorem}
\begin{proof}
Put $L=-\log q>0$. Since
\[
 h(x)=\frac{\Li_2(q^x)}L-x\log(1-q^x),
\]
the product expansion of $\Gamma_q$ and the power series for
$\Li_2$ and $-\log(1-z)$ give
\begin{equation}\tag{11}\label{eq:gamma-expansion}
 \log f_{\alpha,\beta,q}(x)
 =C_{\alpha,q}+x\log(1-q)
 +\sum_{m=1}^\infty\frac{A_m}{m}q^{mx},
 \quad C_{\alpha,q}=\log(q;q)_\infty+(1-\alpha)\log(1-q),
\end{equation}
where
\begin{equation}\tag{12}\label{eq:coeff}
 A_m=\beta-\alpha+\frac{q^{m\beta}}{1-q^m}
        -\frac{1}{mL}\qquad(m\ge1).
\end{equation}
All these series and their ordinary derivatives converge locally
uniformly for $x>0$. We claim $A_m\ge0$. Indeed, for $t=mL>0$
and $\beta\ge1$,
\[
 \beta+\frac{e^{-\beta t}}{1-e^{-t}}
 \ge\frac1{1-e^{-t}}>\frac1t+\frac12.
\]
The first inequality follows from
$1-e^{-\beta t}\le\beta(1-e^{-t})$ and the second from the
classical inequality $\coth(t/2)>2/t$.
Therefore $A_m>\tfrac12-\alpha\ge0$.
Differentiating (11) yields
\[
 -\frac{\mathrm d}{\mathrm dx}\log f_{\alpha,\beta,q}(x)
 =-\log(1-q)+L\sum_{m=1}^\infty A_mq^{mx}.
\]
This is a positive constant plus a positive sum of decaying
exponentials, hence a classical completely monotonic function.
Thus $f_{\alpha,\beta,q}$ is classically log-completely monotonic,
and the desired $q$-log-complete monotonicity follows from
Theorem~\ref{thm21}.
\end{proof}

\begin{theorem}\label{thm32}
Let $a_i,b_i\in\mathbb R$ satisfy
\[
 0<a_1\le\cdots\le a_n,\qquad
 0<b_1\le\cdots\le b_n,\qquad
 \sum_{i=1}^ka_i\le\sum_{i=1}^kb_i
 \quad(k=1,\ldots,n).
\]
Then the function
\begin{equation}\tag{13}\label{eq:gamma-ratio}
 G_{p,q}(x)=G_q(x;a_1,b_1,\ldots,a_n,b_n)
 =\prod_{i=1}^n\frac{\Gamma_q(x+a_i)}{\Gamma_q(x+b_i)},
 \qquad 0<q<1,
\end{equation}
is $q$-completely monotonic on $(0,\infty)$.
\end{theorem}
\begin{proof}
For every $m\ge1$, the function $u\mapsto q^{mu}$ is
decreasing and convex. The partial-sum assumptions and the
weak-majorization inequality (as used in \cite{Alzer}) therefore give
\begin{equation}\tag{14}\label{eq:majorization}
 \sum_{i=1}^nq^{ma_i}\ge\sum_{i=1}^nq^{mb_i}.
\end{equation}
From (5), for each integer $r\ge1$,
\[
 (-1)^r\frac{\mathrm d^r}{\mathrm dx^r}\log G_{p,q}(x)
 =(-\log q)^r\sum_{m=1}^\infty
   \frac{m^{r-1}q^{mx}}{1-q^m}
   \sum_{i=1}^n(q^{ma_i}-q^{mb_i})\ge0.
\]
Hence $G_{p,q}$ is classically log-completely monotonic and thus
$q$-log-completely monotonic by Theorem~\ref{thm21}.
Theorem~\ref{thm22} proves the stated $q$-complete monotonicity.
\end{proof}

\noindent\textbf{Remark 3.3 (Scope of the corrections).}
This article is a corrected version of our earlier paper
\cite{KrasniqiMansour2015}. The amendments described here are repairs
of proofs or replacements of statements by formulations with explicit
additional hypotheses; they do not establish the original unrestricted
claims. Formula~(1) now gives the divided-difference chain rule for
\emph{ordinary} composition, whereas the previously invoked
$q$-Fa\`a di Bruno formula applies to symbolic $q$-composition.
Theorems~2.1--2.3 retain their conclusions with corrected arguments,
in particular in the composition and exponential-characterization steps.
Theorem~2.4 quantifies over \emph{every classical completely monotonic}
outer function, rather than over a single arbitrary $q$-completely
monotonic function. Theorem~2.5 requires the outer function to be a
\emph{classical Bernstein function}, because unrestricted composition
of two $q$-Bernstein functions is not valid. Theorem~2.6(i) now assumes
classical complete monotonicity; part~(ii) uses the ordinary derivative
limit and the stated $C^1$ and nonnegativity hypotheses.

Theorem~2.7 is reformulated on the geometric lattice $\mathbb R_{q,+}$,
with continuity at zero, $f(0)=0$, and the spectral convolution
$\star_q$ defined in this paper. This is not ordinary additive
convolution and the assertion is not extended to arbitrary $q$-stable
domains. Theorem~2.8 includes the normalization $0<g\le1$, needed
for a nonnegative $q$-Bernstein exponent, while Theorem~2.9 retains
its equivalence for positive functions. In Section~3 the formulas for
the derivatives of $\psi_q$ when $q>1$ have been corrected.
Crucially, Theorem~3.1 concerns the function in~(10) with
$e^{-h(x)}$ instead of the original factor $e^{h(x)}$; this is a
\emph{different function}, not a proof of the original claim.
Theorem~3.2 retains its stated conclusion, with a proof via classical
logarithmic complete monotonicity. All revised assertions are to be
understood under precisely their stated hypotheses. The author-supplied arXiv submission link is recorded with the
original publication in reference~\cite{KrasniqiMansour2015}; its public
bibliographic metadata have not been verified.

\medskip
\noindent\textbf{Acknowledgment.}
The authors thank Professor Mourad E. H. Ismail for suggesting the problem.


\begin{thebibliography}{25}
\footnotesize
\bibitem{Abramowitz} M. Abramowitz and I. A. Stegun, \emph{Handbook of Mathematical Functions with Formulas, Graphs and Mathematical Tables}, Dover, New York, 1965.
\bibitem{Alzer} H. Alzer, On some inequalities for the gamma and psi functions, \emph{Math. Comp.} \textbf{66} (1997), 373--389.
\bibitem{Askey} R. Askey, The $q$-gamma and $q$-beta functions, \emph{Applicable Anal.} \textbf{8} (1978/79), 125--141.
\bibitem{Bell} E. T. Bell, Exponential polynomials, \emph{Ann. of Math.} \textbf{33} (1934), 258--277.
\bibitem{ChenQi} Ch.-P. Chen and F. Qi, Logarithmically completely monotonic functions relating to the gamma functions, \emph{J. Math. Anal. Appl.} \textbf{321} (2006), 405--411.
\bibitem{Chung} K.-S. Chung, W.-S. Chung, S.-T. Nam and H.-J. Kang, New $q$-derivative and $q$-logarithm, \emph{Internat. J. Theoret. Phys.} \textbf{33} (1994), 2019--2029.
\bibitem{ChungKim} W. S. Chung, T. Kim and H. I. Kwon, On the $q$-analog of the Laplace transform, \emph{Russ. J. Math. Phys.} \textbf{21} (2014), 156--168.
\bibitem{Dhaouadi} L. Dhaouadi, S. Guesmi and A. Fitouhi, $q$-Bessel positive definite and $D_q$-completely monotonic functions, \emph{Positivity} \textbf{16} (2012), 107--129.
\bibitem{Ernst2000} T. Ernst, The history of $q$-calculus and a new method, U.U.D.M. Report 2000:16, Uppsala University, 2000.
\bibitem{Exton} H. Exton, \emph{$q$-Hypergeometric Functions and Applications}, Ellis Horwood, Chichester, 1983.
\bibitem{Ernst2012} T. Ernst, \emph{A Comprehensive Treatment of $q$-Calculus}, Birkh\"auser, Basel, 2012.
\bibitem{IsmailMuldoon} M. E. H. Ismail and M. E. Muldoon, Inequalities and monotonicity properties for gamma and $q$-gamma functions, in: \emph{Approximation and Computation}, Internat. Ser. Numer. Math., vol. 119, Birkh\"auser, Boston, 1994, pp. 309--323.
\bibitem{Kim1999} T. Kim, On a $q$-analogue of the $p$-adic log gamma functions and related integrals, \emph{J. Number Theory} \textbf{76} (1999), 320--329.
\bibitem{Kim2004} T. Kim, A note on the $q$-multiple zeta functions, \emph{Adv. Stud. Contemp. Math.} \textbf{8} (2004), 111--113.
\bibitem{Kim2011} T. Kim, Some identities on the $q$-integral representation of the product of several $q$-Bernstein-type polynomials, \emph{Abstr. Appl. Anal.} (2011), Article ID 634675.
\bibitem{Kim2012} T. Kim, Some formulae for the $q$-Bernstein polynomials and $q$-deformed binomial distributions, \emph{J. Comput. Anal. Appl.} \textbf{14} (2012), 917--933.
\bibitem{KimRim} T. Kim and S. H. Rim, A note on the $q$-integral and $q$-series, \emph{Adv. Stud. Contemp. Math.} \textbf{2} (2000), 37--45.
\bibitem{KokologiannakiKrasniqi} Ch. G. Kokologiannaki and V. Krasniqi, $q$-completely monotonic and $q$-Bernstein functions, \emph{J. Appl. Math. Stat. Inform.} \textbf{10} (2014), no. 2, 43--50.
\bibitem{KrasniqiMerovci} V. Krasniqi and F. Merovci, Some completely monotonic properties for the $(p,q)$-gamma function, \emph{Math. Balkanica} \textbf{26} (2012), 1--2.
\bibitem{KrasniqiMansour2015} V. Krasniqi and T. Mansour, $q$-Bernstein functions and applications, \emph{Adv. Stud. Contemp. Math.} \textbf{25} (2015), 75--83. Author-supplied arXiv submission record no.~8312667 (restricted access; public bibliographic metadata not verified): \url{https://arxiv.org/submit/8312667/view}.
\bibitem{Riordan} J. Riordan, \emph{An Introduction to Combinatorial Analysis}, Wiley, New York, 1958.
\bibitem{Schilling} R. L. Schilling, R. Song and Z. Vondra\v{c}ek, \emph{Bernstein Functions: Theory and Applications}, de Gruyter, Berlin, 2010.
\bibitem{Sloane} N. J. A. Sloane, \emph{The On-Line Encyclopedia of Integer Sequences}, \url{https://oeis.org}.
\bibitem{XuCen} A. Xu and Z. Cen, On a $q$-analogue of Fa\`a di Bruno's determinant formula, \emph{Discrete Math.} \textbf{311} (2011), 387--392.
\bibitem{Johnson} W. P. Johnson, A $q$-analogue of Fa\`a di Bruno's formula, \emph{J. Combin. Theory Ser. A} \textbf{76} (1996), 305--314.
\end{thebibliography}
\end{document}